\documentclass[a4paper,10pt]{article}

\usepackage[utf8]{inputenc}
\usepackage[english]{babel}
\usepackage{euscript}

\usepackage{graphicx}
\usepackage{amsmath}
\usepackage{amssymb}

\usepackage{pictexwd,dcpic}
\usepackage[all]{xy}

\usepackage{authblk}

\def\quotient#1#2{%
    \raise1ex\hbox{$#1$}\Big/\lower1ex\hbox{$#2$}%
}

\newtheorem{definition}{Definition}
\newtheorem{lemma}{Lemma}
\newtheorem{theorem}{Theorem}

\newenvironment{proof}{\par\addvspace\topsep\noindent
{\bf Proof:} \ignorespaces }{\qed}
\newcommand{\qed}{\hspace*{\fill}$\Box$\ifmmode\else\par\addvspace\topsep\fi}

\title{Counting sub-multisets of fixed cardinality}

\author[1]{Sebastiano Ferraris\thanks{s.ferraris@ucl.ac.uk}}
\author[1]{Alex Mendelson\thanks{a.mendelson@ucl.ac.uk}}
\author[2]{Gerardo Ballesio\thanks{g.ballesio@gmail.com}}
\author[1]{Tom~Vercauteren\thanks{t.vercauteren@ucl.ac.uk}}
\affil[1]{Translational Imaging Group,  Centre~for~Medical~Image~Computing, Wolfson~House, 4~Stephenson~Way, University~College~London.}
\affil[2]{Independent researcher.}

\begin{document}

\maketitle

\begin{abstract}

    This report presents an expression for the number of a multiset's sub-multisets of a given cardinality as a function of the multiplicity of its elements. This is also the number of distinct samples of a given size that may be produced by sampling without replacement from a finite population partitioned into subsets, in the case where items belonging to the same subset are considered indistinguishable. Despite the generality of this problem, we have been unable to find this result published elsewhere.

\bigskip
\noindent \textbf{Keywords:} enumerative combinatorics, multiset, inclusion exclusion principle, constrained $k$-resolutions, cardinality of the support of the multivariate hypergeometric distribution.

\end{abstract}



\section{Introduction}

A multiset is a generalisation of a set that allows elements to appear an integer number of times, rather than being simply present or absent \cite{bona}. The number of times that an element appears in a multiset is termed its multiplicity. One multiset may be considered a sub-multiset of a second when it does not contain any element with a greater multiplicity. The cardinality of a multiset is the sum of its elements' multiplicities, the number of its elements is a distinct quantity called the dimension. \emph{In this report, we provide a formula for the number of sub-multisets of a given multiset that have a specified cardinality.} 

A concrete example of a multiset is found in a packet of candies distinguishable only by their colour. In this case, each colour of candy corresponds to a distinct element of the multiset whose multiplicity is simply the number of candies of that colour that are present. To ask how many distinct handfuls of a given size may be obtained from the packet is to ask how many multisets of a given cardinality exist that are sub-multisets of the multiset represented by the packet. 

To our surprise, even after extensive searching, we were not able to find an answer to this question in the literature. While the number may be found through exhaustive enumeration of the type described in \cite{hage}, this will quickly become impractical as the dimensions and multiplicities of the multisets increase. In this report, we use a proof based on the $k$-resolutions of $n$ and the inclusion-exclusion principle to derive an expression of this quantity as the sum of binomial coefficients presented in equation (\ref{eq:concise_final}). A concise proof is provided in section \ref{se:concise}, while a longer, step-by-step derivation is provided in sections \ref{se:sampling_example} to \ref{se:inc_exc}. 

Python code to verify the provided formula is available online at\\ \verb#https://github.com/SebastianoF/counting_sub_multisets.git#.

\section{Sampling from a multiset} \label{se:sampling_example}
We begin with a simple multiset that possesses 2 elements with an equal multiplicity of 5:
\begin{align*}
S = \{b,b,b,b,b,w,w,w,w,w \}\,.
\end{align*}
In our candy analogy, this corresponds to a packet containing 5 black candies and 5 white ones. In this case, if we want to know how many handfuls of 5 may be produced\footnote{by drawing without replacement, as is usually the case with candies}, we can simply enumerate them, as has been done in the table below. Here, each row represents a handful, and the black and white circles represent the candies of the corresponding colors.
\begin{center}
	\begin{tabular}{  c | c  c  c  c  c  }
        1 ~& ~~$\bullet$  & $\bullet$ & $\bullet$ &  $\bullet$ & $\bullet$~ \\
		\hline    
        2 ~& ~~$\bullet$  & $\bullet$ & $\bullet$ &  $\bullet$ & $\circ$ \\
		\hline    
        3 ~& ~~$\bullet$  & $\bullet$ & $\bullet$ &  $\circ$ & $\circ$ \\
		\hline    
        4 ~& ~~$\bullet$  & $\bullet$ & $\circ$ &  $\circ$ & $\circ$ \\
		\hline    
        5 ~& ~~$\bullet$  & $\circ$ & $\circ$ &  $\circ$ & $\circ$ \\
		\hline    
        6 ~& ~~$\circ$  & $\circ$ & $\circ$ &  $\circ$ & $\circ$ \\
	\end{tabular}
\end{center}
\vspace{0.1in}
\noindent
The number of black candies must be an integer between 0 and 5, this also specifies the number of white candies, and therefore a total of 6 handfuls are possible. Generalising this, let the multiplicities of the elements (black and white candies) be denoted $a_1$ and $a_2$ respectively, the cardinality of the multiset (packet) be denoted $N = a_1 + a_2$, and the cardinality of the sub-multisets (handfuls) be denoted $n$. Providing $n$ is not greater than $a_1$ or $a_2$, the multiplicity of the first element in the sub-multiset may range between 0 and $n$, so there will be $n+1$ possible outcomes. 


If we allow the size of the handful to exceed the number of candies of a given colour, there is an additional constraint we must consider: it is now possible to ``run out of'' one colour of candy as the handful is drawn. If, instead of the original packet, we drew a handful of 5 from a smaller packet containing only 3 black candies and 4 white ones, we would no longer be able to produce the handfuls numbered 1, 2 and 6 in the table above; the number of black candies we draw must be between $n - a_2 = 1$ and $a_1 = 3$. Generalising again, a multiset with two elements of multiplicities $a_1$ and $a_2$ has the following number of sub-multisets with cardinality $n$:
\begin{equation*}
    r = \min(n, a_1) - \max( 1, n - a_2 ) + 2\,.
\end{equation*}

While it has been straightforward to solve this problem in the case of two elements, when we increase the number of colours to a number beyond two, this can no longer be done so intuitively. For example, we might ask how many distinct handfuls of 12 may be drawn from a packet containing 5 blue candies, 9 green ones, and 14 red ones.\footnote{
    If we were to draw handfuls at random, their probabilities would be described by the multivariate hyper-geometric distribution \cite{berkopec}. The number of possible handfuls corresponds to the cardinality of its support.
}
To answer this type of question, we must reformulate the problem in terms of integer resolutions.

%
%


\section{Reformulating the problem with $k$-resolutions}

We consider the multiset $S$ of cardinality $N$ and dimension $k$ where each element $s_j$ has a multiplicity  of $a_j$:
\begin{align*}
S = \{ 
\underbrace{s_1,s_1, ..., s_1}_{a_1\text{-instances}},
\underbrace{s_2, s_2, ..., s_2}_{a_2\text{-instances}},
...,
\underbrace{s_k,s_k, ..., s_k}_{a_k\text{-instances}}
  \} \,,
\end{align*}
where 
\begin{align*}
 a_{1} + a_{2} + \cdots + a_{k} = N\,.
\end{align*}
Providing some ordering is available for the elements, the sequence of their multiplicities $(a_1, a_2, ..., a_k)$ is all that is required to uniquely specify $S$. 
Any sub-multiset of $S$ with cardinality $n$ ($n \leq N$) can be similarly specified by a sequence $(x_1, x_2, ..., x_k)$ which must satisfy the constraints
\begin{align} 
    \label{eq:simple_kresolution}
x_{1} + x_{2} + \cdots + x_{k} = n\,, \qquad  0 \leq x_{j} \leq a_j\,.
\end{align}
Sequences $(x_1, x_2, ..., x_k)$ that satisfy only $x_{1} + x_{2} + \cdots + x_{k} = n$ are known as $k$-resolutions of $n$, and there are established techniques to count them (see \cite{calcolod} or \cite{knuth}, for example). 
The strategy we use to count sequences that satisfy (\ref{eq:simple_kresolution}) is based on their correspondence to the $k$-resolutions of $n$ that satisfy the additional constraints $0 \leq x_{j} \leq a_j $.

\subsection*{$k$-resolutions of $n$}

We denote the number of the $k$-resolutions of $n$
\begin{align}\label{eq:unconstrained}
	R_k^n = \arrowvert \lbrace (x_{1} , x_{2} , \dots , x_{k})  \mid x_{1} + x_{2} + \dots + x_{k} = n, \phantom{z} x_{i} \geq 0  \rbrace \arrowvert\,.
\end{align}
For any choice of $k$ and $n$, $R_k^n $ can be found by considering a row of $n$ indistinguishable candies and $k - 1$ dividers that can be placed to split the row of candies into $k$ subsets. Here, the number of candies in the $j^{th}$ subset corresponds to the integer value of the summand $x_j$. 
For example, the $3$-resolution of $8$ given by $3 + 1 + 4 = 8$ is represented
\begin{center}
	\begin{tabular}{  c  c  c | c | c  c  c  c  }
		$\bullet$  & $\bullet$ & $\bullet$ &  $\bullet$ & $\bullet$ & $\bullet$ &  $\bullet$ & $\bullet$ 
	\end{tabular}
\end{center}
\vspace{0.1in}

\noindent
With this representation, the number of permutations of the full set of candies and dividers is $(n+k-1)!$. Because we consider all candies and dividers indistinguishable, this number will overestimate the number of distinguishable arrangements by a factor of $n! \times (k-1)!\:$. Correcting for this, we obtain
\begin{align}\label{eq:card_k_resolution}
    R_k^n = \frac{(n+k-1)!}{n!(k-1)!} = \binom{ n + k - 1}{ k - 1} = \binom{ n + k - 1}{ n }\,.
\end{align}
This is the number of sub-multisets of cardinality $n$ that a multiset has when $n$ is smaller than all $a_j$.

\subsection*{$k$-resolutions of $n$ with lower constraints }

Before considering the constraints of the form $0 \leq x_{j} \leq a_j$ that occur in our original problem, it is instructive to consider another type of constraints: those of the form $x_j \geq a_j \geq 0$. We shall call these \textit{lower constraints}. The number of $k$-resolutions that satisfy the set of lower constraints specified by the sequence $(a_1,a_2,\dots,a_k)$ is given
\begin{align*}
R_{ (a_1,a_2, \dots , a_k )^{\downarrow} }^{n} = \arrowvert \lbrace (x_{1} , x_{2} , \dots , x_{k})  \mid x_{1} + x_{2} + \dots + x_{k} = n, \phantom{z} x_{i} \geq a_i  \rbrace \arrowvert\,.
\end{align*}
To determine this quantity, we perform the substitution $y_j = x_j - a_j$ for each of the $x_j$:
\begin{align*}
x_{1} + x_{2} + \dots + x_{k} = n  & & x_{i} \geq a_{i} \\
x_{1} - a_{1} + x_{2} - a_{2}  + \dots + x_{k} - a_{k} = n - \sum_{j=1}^{k} a_{j}  & & x_{i} - a_{i} \geq 0 \,.
\end{align*}
This produces the equivalent equation
\begin{align}
    \label{eq:cond_k_resolution_lower}
y_{1} + y_{2}  + \dots + y_{k} = n - \sum_{j=1}^{k} a_{j}  & &  y_{i} \geq 0 \qquad y_{i} = x_{i} - a_{i}\,.
\end{align}
Noting that $n - \sum_{j=1}^{k} a_{j}$ is not negative, we can see that this equation has the same form as equation (\ref{eq:unconstrained}). By equation (\ref{eq:card_k_resolution}), the number of solutions to equation (\ref{eq:cond_k_resolution_lower}) is $R_k^{n - \sum_{j=1}^{k} a_{j} }$, and therefore
\begin{align}     \label{eq:lowercon}
R_{ (a_1,a_2, \dots , a_k )^{\downarrow} }^{n} = \binom{n - \sum_{j=1}^{k} a_{j}  + k - 1}{k-1} = \binom{n - \sum_{j=1}^{k} a_{j}  + k - 1}{n - \sum_{j=1}^{k} a_{j}}\,.
\end{align}

\subsection*{$k$-resolutions with upper constraints}
We now consider the case that corresponds to our original problem. Here, each of the $x_j$ satisfies a constraint of the form $x_{j} \leq a_j$. We denote the number of \textit{upper constrained} $k$-resolutions of $n$
\begin{align}\label{eq:cond_k_resolution_upper} 
R_{ (a_1,a_2, \dots , a_k)^{\uparrow} }^{n} = \arrowvert \lbrace (x_{1} , x_{2} , \dots , x_{k})  \mid x_{1} + x_{2} + \dots + x_{k} = n, 0 \leq x_{i} \leq a_i  \rbrace \arrowvert\,.
\end{align}
The previous strategy can not be adapted so easily to this new case. We could make the substitution
\begin{align*}
x_{1} + x_{2} + \dots + x_{k} = n  & & x_{i} \leq a_{i} \\
x_{1} + x_{2} + \dots + x_{k} = n  & & a_{i} - x_{i} \geq 0 \\
a_{1} - x_{1} + a_{2} -x_{2} + \dots + a_{k} -x_{k} = \sum_{j=1}^{k} a_{j} - n  & & a_{i}- x_{i} \geq 0 \\
\end{align*}
to obtain the form 
\begin{align}
    \label{eq:wrongsub}
y_{1} + y_{2}  + \dots + y_{k} = \sum_{j=1}^{k} a_{j} - n    & &  y_{i} =  a_{i} - x_{i}\,.     
\end{align}
It might seem like we could apply the same reasoning we used for equation (\ref{eq:cond_k_resolution_lower}) to produce the result 
\begin{align}
	\label{eq:wrong_answer}
	R_{ (a_1,a_2, \dots , a_k)^{\uparrow} }^{n} = \binom{\sum_{j=1}^{k} a_{j} - n   + k - 1}{k-1} = \binom{\sum_{j=1}^{k} a_{j} - n   + k - 1}{\sum_{j=1}^{k} a_{j} - n } \,,
\end{align}
but \textit{this formula is not correct}. Unlike those of equation (\ref{eq:cond_k_resolution_lower}), the corresponding $y_j$ of equation (\ref{eq:wrongsub}) may take negative values.

This can be seen in the counter example that follows. If we consider the problem specified by $n = 5$, $k = 3$ and $(a_1,a_2,a_3) = (2,3,3)$, then the possible resolutions are the 9 that follow in the form $[x_1,x_2,x_3]$:
\begin{align*}
&[0, 2, 3], [0, 3, 2], [1, 1, 3], \\
&[1, 2, 2], [1, 3, 1], [2, 0, 3],\\ 
&[2, 1, 2], [2, 2, 1],[2, 3, 0]\,.
\end{align*}
As can be seen below, equation (\ref{eq:wrong_answer}) would suggest that there are $10$:
\begin{align*}
\binom{8 - 5   + 3 - 1}{3 - 1} = 10\,.
\end{align*}
This has happened because we have also counted the invalid solution
\begin{align*}
[-1, 3, 3]\,.
\end{align*}

To find $R_{ (a_1,a_2, \dots , a_k)^{\uparrow} }^{n}$, the number of upper constrained $k$-resolutions of $n$, we must use a different strategy.


\section{Solution from the inclusion-exclusion principle} \label{se:inc_exc}

We begin again with the equations specifying a $k$-resolution of $n$ as a sequence of integers $x_j$ that satisfy the original constraints:
\begin{align*}
x_{1} + x_{2} + \dots + x_{k} = n\,,  \qquad 0 \leq x_{i} \leq a_{i}\,.
\end{align*}
We consider the following:
\begin{align*}
A &= \lbrace (x_{1} , x_{2} , \dots , x_{k})  \mid x_{1} + x_{2} + \dots + x_{k} = n, \text{ such that } \forall i  \phantom{z} 0 \leq x_{i} \leq a_{i}  \rbrace \\
B &= \lbrace (x_{1} , x_{2} , \dots , x_{k})  \mid x_{1} + x_{2} + \dots + x_{k} = n, \text{ such that } \exists i  \phantom{z} x_{i} \geq a_{i} +1  \rbrace \\
C &= \lbrace (x_{1} , x_{2} , \dots , x_{k})  \mid x_{1} + x_{2} + \dots + x_{k} = n, \text{ such that } \forall i  \phantom{z} x_{i} \geq 0 \rbrace \,.
\end{align*}
Here, $C$ is the full set of $k$-resolutions of $n$, $A$ is the set of the upper constrained resolutions 
(whose cardinality we wish to determine), and $B$ is the set where at least one of the $x_j$ exceeds the shared bounds $a_j$. 
We can see that $A = C \setminus B$ and $\arrowvert A \arrowvert =  \arrowvert C \arrowvert - \arrowvert B \arrowvert $. \\

By (\ref{eq:card_k_resolution}), the cardinality of $C$ is 
\begin{align*}
\arrowvert C \arrowvert = \binom{n+k-1}{k-1}\,.
\end{align*}
To determine $|A|$, we now only need to determine $|B|$. $B$ may be written as the union 
\begin{align*}
B = B_{1}\cup B_{2} \cup \dots \cup B_{k} \,,
\end{align*}
where
\begin{align*}
B_{j} &= \lbrace (x_{1} , x_{2} , \dots , x_{k})  \mid x_{1} + x_{2} + \dots + x_{k} = n,  x_{j} \geq a_{j} +1  \rbrace \,.
\end{align*}
By the inclusion-exclusion principle,
\begin{align*}
\arrowvert  B_{1}\cup B_{2} \cup \dots \cup B_{k} \arrowvert 
&= 
\sum_{m=1}^{k} (-1)^{m+1}\sum_{1\leq i_1 < i_2 < \dots <i_m \leq  k} \arrowvert B_{i_1}\cap B_{i_2} \cap \dots \cap  B_{i_m}  \arrowvert\,. \\
\end{align*}
At this point, we note that the sets $B_{i_1} \cap B_{i_2} \cap \dots \cap  B_{i_m}$ are sets of \textit{lower constrained} $k$-resolutions of the form 
\begin{align*}
B_{i_1} \cap B_{i_2} \cap \dots \cap  B_{i_m} 
&= 
\lbrace (x_{1} , x_{2} , \dots , x_{k})  
\mid 
x_{1} + x_{2} + \dots + x_{k} = n,\, x_{i_l} \geq a_{i_l} + 1\, , l = 1, \dots, m \rbrace \,.
\end{align*}
By equation (\ref{eq:lowercon}), the value of each summand must therefore be

%
\begin{align*}
\arrowvert  B_{i_1} \cap B_{i_2} \cap \dots \cap  B_{i_m}\arrowvert 
&= 
\binom{ n- (a_{i_{1}} + 1 + a_{i_{2}} + 1 +\dots + a_{i_{m}}+1) + k - 1 }{ k-1 }  \\
&= 
\binom{ n- (\sum_{l=1}^{m}a_{i_{l}} +m) + k - 1 }{ k-1 }  \,.
\end{align*}
With this last formula we can finally determine the cardinality of $A$ (i.e., the number of upper constrained k-resolutions):
\begin{align*}
R_{ (a_1,a_2, \dots , a_k)^{\uparrow} }^{n} &=  \arrowvert  A \arrowvert  = \arrowvert C \arrowvert - \arrowvert B \arrowvert \\
&=\binom{n+k-1}{k-1} - 
\sum_{m=1}^{k} (-1)^{m+1}
\sum_{1\leq i_1 < i_2 < \dots <i_m \leq  k} 
\binom{n-\sum_{l=1}^{m} a_{i_{l}} - m + k - 1}{k-1} \,.
\end{align*}
This can be expressed more concisely as
\begin{align}
	R_{ (a_1,a_2, \dots , a_k)^{\uparrow} }^{n} 
	= 
	\sum_{ L \in \mathcal{P}(I_k)} 
	(-1)^{|L|} 
	\binom{ n + k - 1 - |L| - \sum\limits_{i \in L} a_i }{ k-1 }\,,
\end{align}
where $\mathcal{P}(I_k)$ denotes the power set of $I_k$, and binomial coefficients of the form $\binom{\alpha}{\beta}$ with $\alpha < 0$ or $\alpha < \beta$ are taken to be zero in line with their combinatorial meaning.
 


\section{Concise proof} \label{se:concise}

This section provides a summary of the definitions and proofs presented in this report.

\begin{definition}
Given a finite set $A$ of cardinality $k$, a \emph{multiset} $S$ is a collection of elements of $A$ where each element of $A$ can appear zero or more times, and where the order of the elements and appearances does not matter.
The number of occurrences of an element $a \in A$ in $S$ is called its \textit{multiplicity}. The sum of all a multiset's elements' multiplicities is its \textit{cardinality}. The \textit{dimension} of the multiset $S$ is the number of distinct elements that it contains one or more times.
\end{definition}

\begin{definition}
	A \emph{sub-multiset} $S'$ of a multiset $S$ is a multiset whose elements are all also elements of $S$. The elements' multiplicities in $S'$ must be less than or equal to their corresponding multiplicities in $S$.
\end{definition}
We wish to determine the number of multisets of cardinality $n$ that may be considered sub-mulitisets of a multiset $S$ of cardinality $N$ and dimension $k$. The $k$ distinct elements have multiplicities in $S$ specified by the sequence $(a_1,a_2,\dots,a_k)$. This problem can be formulated in terms of constrained resolutions of integers.

\begin{definition}
	A \emph{$k$-resolution of $n$} is an element of the set
	\begin{align*} 
	\lbrace (x_{1} , x_{2} , \dots , x_{k})  \mid x_{1} + x_{2} + \dots + x_{k} = n, \phantom{z} x_{i} \geq 0  \rbrace \,.
	\end{align*}
	The cardinality of this set is denoted
	\begin{align*}
	R_{k}^{n} = \arrowvert \lbrace (x_{1} , x_{2} , \dots , x_{k})  \mid x_{1} + x_{2} + \dots + x_{k} = n, \phantom{z} x_{i} \geq 0  \rbrace \arrowvert\,.
	\end{align*}
	The cardinality of the set of $k$-resolutions with \emph{lower constraints} specified by a sequence $(a_1,a_2,\dots,a_k)$ is
	\begin{align*}
	R_{ (a_1,a_2, \dots , a_k )^{\downarrow} }^{n}  = \arrowvert \lbrace (x_{1} , x_{2} , \dots , x_{k})  \mid x_{1} + x_{2} + \dots + x_{k} = n, 0 \geq x_{i} \geq a_i  \rbrace \arrowvert\,,
	\end{align*}
	while the cardinality of the set of $k$-resolutions with \emph{upper constraints} specified by such a sequence is
	\begin{align*}
	R_{ (a_1,a_2, \dots , a_k)^{\uparrow} }^{n} = \arrowvert \lbrace (x_{1} , x_{2} , \dots , x_{k})  \mid x_{1} + x_{2} + \dots + x_{k} = n, x_{i} \leq a_i  \rbrace \arrowvert\,.
	\end{align*}
\end{definition}
%
The number of sub-multisets is thus equivalent to the number of $k$-resolutions of $n$ with the upper constraints $(a_1, a_2, \dots, a_k)$.

\begin{lemma}
	The cardinality of the set of $k$-resolutions with lower constraints specified by $(a_1,a_2,\cdots,a_k)$ is given
	\begin{align*}
	R_{ (a_1,a_2, \dots , a_k )^{\downarrow} }^{n} 
	= \binom{n - \sum_{j=1}^{k} a_{j}  + k - 1}{k-1} \,.
	\end{align*}
\end{lemma}
\begin{proof}
	
	\begin{align*} 
	R_{ (a_1,a_2, \dots , a_k )^{\downarrow} }^{n}
	&=
	\arrowvert \lbrace (x_{1} , x_{2} , \dots , x_{k})  \mid x_{1} + x_{2} + \dots + x_{k} = n, x_{i} \geq a_i  \rbrace \arrowvert\\
	&=
	\arrowvert \lbrace (x_{1} , x_{2} , \dots , x_{k})  \mid x_{1} - a_{1} + x_{2} - a_{2}  + \dots + x_{k} - a_{k} = n - \sum_{j=1}^{k} a_{j}, x_{i} \geq a_i  \rbrace \arrowvert\\
	&=
	\arrowvert \lbrace (x_{1} , x_{2} , \dots , x_{k})  \mid y_{1} + y_{2}  + \dots + y_{k} = n - \sum_{j=1}^{k} a_{j}, ~  y_{i} \geq 0  \rbrace \arrowvert\\
    &\text{and so using the standard result for unconstrained $k$-resolutions, } \\
    &=\binom{n - \sum_{j=1}^{k} a_{j}  + k - 1}{k-1} \,.
	\end{align*}
\end{proof}

\begin{theorem}
	Indicating  $\{1, 2,\cdots, k\} $ with $I_k$, the cardinality of the set of upper constrained $k$-resolutions is
	\begin{align}
        \label{eq:concise_final}
		R_{ (a_1,a_2, \dots , a_k)^{\uparrow} }^{n} 
		= 
		\sum_{ L \in \mathcal{P}(I_k)} 
		(-1)^{|L|} 
		\binom{ n + k - 1 - |L| - \sum\limits_{i \in L} a_i }{ k-1 }\,,
	\end{align}
	where $\mathcal{P}(I_k)$ is the power set of $I_k$.
\end{theorem}
\begin{proof}
	We define the sets $A$, $B$ and $C$ as follows:
	\begin{align*}
    A &= \lbrace (x_{1} , x_{2} , \dots , x_{k})  \mid x_{1} + x_{2} + \dots + x_{k} = n, \text{ such that } \forall i  \phantom{z} 0 \leq x_{i} \leq a_{i}  \rbrace \\
    B &= \lbrace (x_{1} , x_{2} , \dots , x_{k})  \mid x_{1} + x_{2} + \dots + x_{k} = n, \text{ such that } \exists i  \phantom{z} x_{i} \geq a_{i} +1  \rbrace \\
    C &= \lbrace (x_{1} , x_{2} , \dots , x_{k})  \mid x_{1} + x_{2} + \dots + x_{k} = n, \text{ such that } \forall i  \phantom{z} x_{i} \geq 0 \rbrace \,.
	\end{align*}
    From this, it follows that
	\begin{align*}
	A \cup B = C
	\qquad \qquad 
	A\cap B = \emptyset 
	\qquad \qquad 
	\arrowvert A \arrowvert =  \arrowvert C \arrowvert - \arrowvert B \arrowvert 
	\qquad \qquad 
	\arrowvert C \arrowvert = \binom{n+k-1}{k-1}\,.
	\end{align*}
	The set $B$ can be written as the union of subsets 
	\begin{align*}
	B &= B_{1}\cup B_{2} \cup \dots \cup B_{k}\,,
	\end{align*}
    where
	\begin{align*}
	B_{j} &= \lbrace (x_{1} , x_{2} , \dots , x_{k})  \mid x_{1} + x_{2} + \dots + x_{k} = n, x_{j} \geq a_{j} +1  \rbrace \,.
	\end{align*}
	Due to the inclusion-exclusion principle,
	\begin{align*}
	\arrowvert  B_{1}\cup B_{2} \cup \dots \cup B_{k} \arrowvert 
	&= 
	\sum_{m=1}^{k} (-1)^{m+1}\sum_{1\leq i_1 < i_2 < \dots <i_j \leq  k} \arrowvert B_{i_1}\cap B_{i_2} \cap \dots \cap  B_{i_m}  \arrowvert \,. \\
	\end{align*}
	From the previous lemma, it follows that
	\begin{align*}
	\arrowvert  B_{i_1} \cap B_{i_2} \cap \dots \cap  B_{i_m}\arrowvert 
	&= 
	\binom{ n- (a_{i_{1}} +1 + a_{i_{2}} + 1 +\dots + a_{i_{m}}+1) + k - 1 }{ k-1 }  \,,
	\end{align*}
	and therefore
	\begin{align*}
		R_{ (a_1,a_2, \dots , a_k)^{\uparrow} }^{n} 
		=
		\binom{n+k-1}{k-1} - \sum_{m=1}^{k} (-1)^{m+1}\sum_{1\leq i_1 < i_2 < \dots <i_j \leq  k} \binom{n-\sum_{l=1}^{m} a_{i_{l}} - m + k - 1}{k-1} \,.
	\end{align*}
\end{proof}

\section*{Acknowledgment}

Sebastiano Ferraris is supported by the EPSRC-funded UCL Centre for Doctoral Training in Medical Imaging (EP/L016478/1).\\
Alex Mendelson is supported  by UCL (code ELCX),
a CASE studentship with the EPSRC and GE healthcare, EPSRC
grants (EP/H046410/1, EP/ H046410/1, EP/J020990/1, EP/K005278),
the MRC (MR/J01107X/1), the EU-FP7 project VPH-DARE@IT
(FP7-ICT-2011-9-601055), the NIHR Biomedical Research Unit (Dementia) at UCL and the National Institute for Health Research University College London Hospitals Biomedical Research Centre (NIHR BRC UCLH/UCL High Impact Initiative). \\
We would like to thank Carlo Matriconda and Alberto Tonolo for having inspired this work.



\begin{thebibliography}{20}

\bibitem{berkopec}
Berkopec, Aleš. HyperQuick algorithm for discrete hypergeometric distribution. Journal of Discrete Algorithms 5.2 (2007): 341-347.

\bibitem{bona}
Bona, Miklos. Introduction to enumerative combinatorics. McGraw-Hill Higher Education, 2007.

\bibitem{calcolod}
C. Mariconda, A. Tonolo, \emph{Calcolo Discreto}, Apogeo 2012.

\bibitem{hage}
Hage, Jurriaan. Enumerating submultisets of multisets. Information processing letters 85.4 (2003): 221-226.

\bibitem{knuth}
Patashnik, RL Graham DE Knuth O. Concrete Mathematics-A Foundation for Computer Science. (1989).

\bibitem{ross}
Ross, Sheldon M. Introduction to probability and statistics for engineers and scientists. Academic Press, 2014.







\end{thebibliography}
\end{document}